\documentclass[titlepage,12pt]{article} 
\usepackage{hyperref}
\usepackage[usenames,dvipsnames]{xcolor}
\usepackage[title]{appendix}
\usepackage{amssymb,amsthm,amsmath} 
\usepackage[a4paper]{geometry}
\usepackage[italian,english]{babel}
\usepackage{datetime2}
\usepackage{enumitem}
\usepackage[utf8]{inputenc}

\date{}

\newcommand{\ep}{\varepsilon}
\newcommand{\re}{\mathbb{R}}

\newcommand{\n}{\mathbb{N}}

\newcommand{\KUB}{\operatorname{OB}_{\mathrm{U}}}
\newcommand{\KGB}{\operatorname{OB}_{\mathrm{G}}}
\newcommand{\KPB}{\operatorname{OB}_{\mathrm{P}}}
\newcommand{\X}{\mathbb{X}}
\newcommand{\Y}{\mathbb{Y}}
\newcommand{\A}{\mathbb{A}}

\newtheorem{thm}{Theorem}[section]
\newtheorem{thmbibl}{Theorem}

\newtheorem{rmk}[thm]{Remark}

\newtheorem{defn}[thm]{Definition}
\newtheorem{cor}[thm]{Corollary}

\newtheorem{lemma}[thm]{Lemma}
\newtheorem{open}{Open problem}

\title{Sharp ultimate velocity bounds for the general solution of some linear second order evolution equation with damping and bounded forcing}

\author{Marina Ghisi\vspace{1ex}\\ 
{\normalsize Università degli Studi di Pisa} \\
{\normalsize Dipartimento di Matematica}\\ 
{\normalsize PISA (Italy)}\\
{\normalsize e-mail: \texttt{marina.ghisi@unipi.it}}
\and
Chiara Giraudo\vspace{1ex}\\
{\normalsize Università degli Studi di Trento} \\ 
{\normalsize Dipartimento di Matematica}\\ 
{\normalsize Trento (Italy)}\\
{\normalsize e-mail: \texttt{chiaragiraudo96@gmail.com}}
\and
Massimo Gobbino\vspace{1ex}\\ 
{\normalsize Università degli Studi di Pisa} \\
{\normalsize Dipartimento di Ingegneria Civile e Industriale}\\ 
{\normalsize PISA (Italy)}\\  
{\normalsize e-mail: \texttt{massimo.gobbino@unipi.it}}
\and
Alain Haraux\vspace{1ex}\\ 
{\normalsize Sorbonne Université, Université Paris-Diderot SPC, CNRS, INRIA}, \\
{\normalsize Laboratoire Jacques-Louis Lions,  LJLL, F-75005,
Paris, France.}\\ 
{\normalsize e-mail: \texttt{haraux@ann.jussieu.fr}}}

\begin{document}
\maketitle

\begin{abstract}

We consider a class of linear second order differential equations with damping and external force. We investigate the link between a uniform bound on the forcing term and the corresponding ultimate bound on the velocity of solutions, and we study the dependence of that  bound on the damping and on the ``elastic force''.

We prove three results. First of all, in a rather general setting we show that different notions of bound are actually equivalent. Then we compute the optimal constants in the scalar case. Finally, we extend the results of the scalar case to abstract dissipative wave-type equations in Hilbert spaces. In that setting we obtain rather sharp estimates that are quite different from the scalar case, in both finite and infinite dimensional frameworks.

The abstract theory applies, in particular, to dissipative wave, plate and beam equations.
\vspace{6ex}

\noindent{\bf Mathematics Subject Classification 2010 (MSC2010):} 
35B40, 34D05, 34C11, 35L90.


\vspace{6ex}


\end{abstract}

\clearpage
 
\section{Introduction}

The present paper deals with optimal estimates for the ultimate bound of solutions to a class of evolution problems with bounded source term. 

Our starting point is a paper by W.~S.~Loud~\cite{loud} concerning the second order ordinary differential equation
\begin{equation}
u''(t)+cu'(t)+g(u(t))=f(t),
\label{eqn:loud}
\end{equation}
which has been the object of many researches since the pioneering papers of G.~Duffing  who considered the case where $g$ is a polynomial of degree 3.  Many texts have been written on this special case even recently, see for example~\cite{Bren-Kovac}. Vector generalizations of this equation in both finite and infinite dimensional settings appear naturally, with or without forcing terms, in the context of stability theory for beams, cranes and more recently suspension bridges, compare for example~\cite{Gazzola2, Garrione-Gazzola,Gazzola1}. Before stating our present results, we shall recall now only the researches dealing with the linear or ``single well" nonlinear cases.

\paragraph{\textmd{\textit{Previous results}}}

In~\cite{loud} the following asymptotic bounds are proved.

\begin{thmbibl}[{see~\cite[Theorem~1]{loud}}]\label{thmbibl:loud}

Let us assume that
\begin{itemize}

\item $c$ is a positive real number,

\item  $g:\re\to\re$ is a function of class $C^{1}$ such that $g(0)=0$ and
\begin{equation}
g'(s)\geq b
\qquad
\forall s\in\re
\label{hp:loud}
\end{equation}
for some positive real number $b$,

\item  $f\in L^{\infty}((0,+\infty),\re)$ is a bounded forcing term.

\end{itemize}

Then every solution to (\ref{eqn:loud}) satisfies
\begin{equation}
\limsup_{t\to +\infty}|u(t)|\leq
\min\left\{\frac{1}{b}+\frac{4}{c^{2}},\frac{1}{b}+\frac{4}{c\sqrt{b}}\right\}\|f\|_{L^{\infty}((0,+\infty),\re)},
\label{th:loud-u}
\end{equation}
and
\begin{equation}
\limsup_{t\to +\infty}|u'(t)|\leq
\frac{4}{c}\|f\|_{L^{\infty}((0,+\infty),\re)}.
\label{th:loud-u'}
\end{equation}

\end{thmbibl}

Both estimates can be refined by replacing the $L^{\infty}$ norm of $f$ by the limsup at infinity of $f$, intended (as we are dealing with functions defined up to a negligible set) as
\begin{equation}
\limsup_{t\to +\infty}|f(t)|:=\lim_{T\to +\infty}\|f\|_{L^{\infty}((T,+\infty),\re)}.
\nonumber
\end{equation} 

We also observe that, at least in the linear case where $g(s)=bs$, the left-hand sides of (\ref{th:loud-u}) and (\ref{th:loud-u'}) do not depend on the solution $u(t)$, but just on the forcing term $f(t)$. This follows from the decay to zero of solutions to the corresponding homogeneous equation.

The proof of (\ref{th:loud-u}) and (\ref{th:loud-u'}) provided in~\cite{loud} relies on a delicate geometric argument in the phase space, and it is by no means evident how this argument could be extended to vector equations, and even less to infinite dimensions. In order to overcome this difficulty, more than 50 years later the authors of~\cite{2009-JMPA-FitHar,2013-DCDS-FitHar} tried to recover the same estimates by a purely analytical method based on differential inequalities (energy estimates), obtaining the following result.
\begin{thmbibl}[{see~\cite[Theorem~1.1, Theorem~2.1, and Proposition~2.4]{2013-DCDS-FitHar}}]\label{thmbibl:alain-ODE}

Let us consider equation (\ref{eqn:loud}) under the same assumptions of Theorem~\ref{thmbibl:loud}.

Then every solution satisfies
\begin{equation}
\limsup_{t\to +\infty}|u(t)|\leq
\max\left\{\frac{1}{b},\frac{2}{c\sqrt{b}}\right\}\limsup_{t\to +\infty}|f(t)|,
\label{th:alain-ode-u}
\end{equation}
and
\begin{equation}
\limsup_{t\to +\infty}|u'(t)|\leq
\left\{
\begin{array}{l@{\qquad}l}
\displaystyle\left(\dfrac{2}{c}+\dfrac{1}{\sqrt{b}}\right)\limsup_{t\to +\infty}|f(t)| & 
\mbox{if }c<2\sqrt{b}, \\[3ex]
\displaystyle\dfrac{2}{\sqrt{b}}\limsup_{t\to +\infty}|f(t)| & 
\mbox{if }c\geq 2\sqrt{b}.
\end{array}
\right.
\label{th:alain-ode-u'}
\end{equation}

\end{thmbibl}

As for estimates on $u(t)$, it is not difficult to check that (\ref{th:alain-ode-u}) improves (\ref{th:loud-u}) for all positive values of $b$ and $c$. The best constant for which an ultimate bound of the form (\ref{th:alain-ode-u}) is true was computed in~\cite{2005-AccXL-Haraux} in the case of the linear equation 
\begin{equation}
u''(t)+cu'(t)+bu(t)=f(t).
\label{eqn:scalar}
\end{equation}

The result is the following (see also~\cite{giraudo:tesi}).
\begin{thmbibl}[{see~\cite[Theorem~2.1]{2005-AccXL-Haraux}}]\label{thmbibl:alain-LODE}

Let $b$ and $c$ be positive real numbers.

Then every solution to (\ref{eqn:scalar}) satisfies
\begin{equation}
\limsup_{t\to +\infty}|u(t)|\leq
\left\{
\begin{array}{l@{\qquad}l}
\displaystyle\dfrac{1}{b}\coth\left(\frac{c\pi}{2\sqrt{4b-c^{2}}}\right)\cdot\limsup_{t\to +\infty}|f(t)| & 
\mbox{if }c<2\sqrt{b}, \\[3ex]
\displaystyle\dfrac{1}{b}\limsup_{t\to +\infty}|f(t)| & 
\mbox{if }c\geq 2\sqrt{b},
\end{array}
\right.
\label{th:alain-lode-u}
\end{equation}
and this estimate is sharp for all positive values of $b$ and $c$.

\end{thmbibl}

As for estimates on $u'(t)$, it turns out that (\ref{th:alain-ode-u'}) improves (\ref{th:loud-u'}) only when $c\leq 2\sqrt{b}$. Even worse, when $b$ is fixed, the constant $4/c$ of Loud's estimate (\ref{th:loud-u'}) tends to~0 as $c\to +\infty$, while this in not true in (\ref{th:alain-ode-u'}).

This is quite difficult to interpret. On the one hand, Loud's method does not rely on any conditions on $c$. On the other hand, the optimality of the constant in (\ref{th:alain-lode-u}) suggests that there are actually two distinct regimes. This separation is natural also if we consider the linear equation (\ref{eqn:scalar}), whose solutions in the homogeneous case $f(t)\equiv 0$ are oscillatory when $c<2\sqrt{b}$, and non-oscillatory when $c\geq 2\sqrt{b}$.

A natural generalization of (\ref{eqn:scalar}) is the abstract linear evolution equation of the form  
\begin{equation}
u''(t)+Bu'(t)+Au(t)=f(t),
\label{eqn:PDE-AB}
\end{equation}
where $A$ and $B$ are suitable operators defined in a Hilbert space. This equation was considered in~\cite{2009-JMPA-FitHar} for $B$ a multiple of the identity and $A$ a possibly nonlinear operator, and then in the general linear case in~\cite{2013-JFA-AloHar}, where the following ultimate bounds are proved.

\begin{thmbibl}[{see~\cite[Theorem~2.1]{2013-JFA-AloHar}}]\label{thmbibl:alain-PDE}

Let us consider equation (\ref{eqn:PDE-AB}) in a separable Hilbert space $H$. Let us assume that
\begin{itemize}

\item $A$ is a self-adjoint operator with dense domain $D(A)$, and satisfying the coercivity assumption 
\begin{equation}
\langle Av,v\rangle\geq b\|v\|_{H}^{2}
\qquad
\forall v\in D(A)
\label{hp:A-coercive}
\end{equation}
for some positive real number $b$,

\item $B:D(A^{1/2})\to D(A^{-1/2})$ is a self-adjoint operator such that
\begin{equation}
c\|v\|_{H}^{2}\leq\langle Bv,v\rangle\leq C\|A^{1/2}v\|_{H}^{2}
\qquad
\forall v\in D(A^{1/2})
\label{hp:BA}
\end{equation}
for some positive real numbers $c$ and $C$,

\item  $f\in L^{\infty}((0,+\infty),H)$.

\end{itemize}

Then every weak solution 
\begin{equation}
u\in C^{0}\left((0,+\infty),D(A^{1/2})\right)\cap C^{1}\left((0,+\infty),H)\strut\right)
\label{defn:weak-sol}
\end{equation}
satisfies
\begin{equation}
\limsup_{t\to +\infty}\|A^{1/2}u(t)\|_{H}\leq
\max\left\{\frac{\sqrt{3C}}{\sqrt{c}},\frac{3}{c\sqrt{2}}\right\}\limsup_{t\to +\infty}\|f(t)\|_{H},
\label{th:alain-pde-u}
\end{equation}
and
\begin{equation}
\limsup_{t\to +\infty}\|u'(t)\|_{H}\leq
\max\left\{\frac{\sqrt{3C}}{\sqrt{c}},\frac{3}{c\sqrt{2}}\right\}\limsup_{t\to +\infty}\|f(t)\|_{H}.
\label{th:alain-pde-u'}
\end{equation}

\end{thmbibl}

Apparently the coercivity constant $b$ does not appear in the final estimates (\ref{th:alain-pde-u}) and (\ref{th:alain-pde-u'}), but its presence is actually ``hidden'' in the left-hand side of (\ref{th:alain-pde-u}), and in the choice of the constant $C$. For example, if we restrict to the scalar case where $H=\re$, $Bv=cv$ and $Av=bv$, then (\ref{hp:BA}) holds true with $C:=c/b$, and hence the estimate on the velocity reads as
\begin{equation}
\limsup_{t\to +\infty}|u'(t)|\leq
\max\left\{\frac{\sqrt{3}}{\sqrt{b}},\frac{3}{c\sqrt{2}}\right\}\limsup_{t\to +\infty}|f(t)|.
\nonumber
\end{equation}

Compared with (\ref{th:loud-u'}), this estimate is better for small values of $c$, but worse for large values of $c$, and again the constant does not tend to~$0$ as $c\to +\infty$.

\paragraph{\textmd{\textit{Our results}}}

Since the known estimates on $u(t)$ are optimal, at least in the scalar case (see Theorem~\ref{thmbibl:alain-LODE} above), in this paper we focus on estimates on derivatives, which are arguably harder, and we address three questions.

In the first part of the paper we compare different notions of bound. To this end, we consider equations of the form (\ref{eqn:scalar}) or (\ref{eqn:PDE-AB}) with forcing term $f(t)$ that is defined and bounded on the whole real line, and not just for positive times. Under these assumptions, these equations admit a \emph{unique} global solution that is bounded for all times, positive and negative. In the scalar case this solution satisfies an inequality of the form 
\begin{equation}
|u'(t)|\leq K_{GB}\|f\|_{L^{\infty}(\re,\re)}
\qquad
\forall t\in\re,
\label{defn:GB-intro}
\end{equation}
and similarly in the vector case. The first result of this paper is that the optimal constant $K_{GB}$ for which \emph{global bounds} such as (\ref{defn:GB-intro}) hold true (for the unique global solution) coincides with the optimal constant for which \emph{ultimate bounds} such as (\ref{th:loud-u'}) or (\ref{th:alain-ode-u'}) hold true (for all solutions defined for positive times). In Theorem~\ref{thm:abstract} this result is proved in a general framework that contains as a special case both the estimates on $u(t)$ and the estimates on $u'(t)$, and applies to large classes of evolutions problems including (\ref{eqn:scalar}) and (\ref{eqn:PDE-AB}).

In the second part of the paper we focus on the scalar linear equation (\ref{eqn:scalar}), and in Theorem~\ref{thm:scalar} we compute the optimal constant that appears in ultimate or global bounds. From the explicit computation we deduce that this constant decreases both with respect to $b$ and with respect to $c$, and of course it tends to~0 as $c\to +\infty$. The monotonicity properties are quite delicate, and for this reason we suspect it could be difficult to extend them to vector equations (the monotonicity with respect to the ``elastic term'' is even false in the vector case, as we show in Corollary~\ref{cor:no-monotone}).

In the third part of the paper we address the vector equation
\begin{equation}
u''(t)+cu'(t)+Au(t)=f(t),
\label{eqn:vector}
\end{equation}
namely equation (\ref{eqn:PDE-AB}) in the special case where the damping is a positive multiple of the velocity. In Theorem~\ref{thm:KcA-upper} we prove bounds on the velocity with a constant that tends to~0 as $c\to +\infty$. To be more precise, in finite dimension $d$ the constant is always less than $2\sqrt{d}/c$, while in infinite dimensions it is always less that $O((\log c)^{1/2}/c)$. 

Finally, in Theorem~\ref{thm:KcA-lower} we show that the term $\sqrt{d}$ is essential in finite dimension, and that in infinite dimensions the correction $(\log c)^{1/2}$ is essential if the eigenvalues of the operator are unbounded but grow at most exponentially (this case includes many operators that are important in the applications, for example the Dirichlet Laplacian, as shown in Remark~\ref{rmk:Laplacian}). The need of this unexpected correction could explain why it was so hard to extend Loud's result to partial differential equations.

\paragraph{\textmd{\textit{Structure of the paper}}}

This paper is organized as follows. In section~\ref{sec:equivalence} we prove that different bounds on the velocity hold true with the same optimal constants. In section~\ref{sec:scalar} we present optimal bounds for the scalar equation (\ref{eqn:scalar}), and we discuss the dependence of the optimal constant on the parameters $b$ and $c$. In section~\ref{sec:vector} we address the vector equation (\ref{eqn:vector}), and we prove our estimate from above and from below for the optimal constants. Finally, in section~\ref{sec:open} we present some future perspectives and open problems.


\setcounter{equation}{0}
\section{Different notions of optimal bounds}\label{sec:equivalence}

\subsection{Functional setting and definitions}

In this section we consider the following functional setting:
\begin{itemize}

\item  $\X$ is a (real) Banach space,

\item  $\Y$ is a linear subspace with the norm inherited from $\X$,

\item  $p$ is a seminorm in $\X$ that is continuous with respect to the norm of $\X$, namely there exists a real number $P$ such that
\begin{equation}
p(x)\leq P\|x\|_{\X}
\qquad
\forall x\in\X,
\label{hp:p-P}
\end{equation} 

\item  $\A$ is the infinitesimal generator of a linear semigroup $S(t)$ on $\X$, which we assume to be exponentially damped in the sense that there exist two positive real numbers $C$ and $\delta$ such that
\begin{equation}
\|S(t)x\|_{\X}\leq Ce^{-\delta t}\|x\|_{\X}
\qquad
\forall t\geq 0,\quad\forall x\in\X.
\label{hp:St}
\end{equation} 

\end{itemize}

We consider the abstract evolution equation
\begin{equation}
U'(t)+\A U(t)=F(t),
\label{eqn:abstract}
\end{equation}
where $F:\re\to\X$ or $F:(0,+\infty)\to\X$ is a suitable bounded forcing term. When $F$ is globally defined we look for solutions that are globally bounded. When $F$ is defined only for positive times, we consider solutions that satisfy a suitable initial condition 
\begin{equation}
U(0)=U_{0}.
\label{eqn:datum}
\end{equation}

In the following statement we summarize some well-known results concerning existence of such solutions, and their representation in terms of the semigroup (see~\cite{alain:book-LNM,1998-MACo-Haraux}).

\begin{thmbibl}[Existence of bounded solutions of different types]\label{thmbibl:existence}

Let us consider the evolution equation (\ref{eqn:abstract}) in the functional setting described above. 

Then the following statements hold true.
\begin{enumerate}
\renewcommand{\labelenumi}{(\arabic{enumi})}

\item  \emph{(Bounded solutions for positive times).} For every source $F\in L^{1}_{\mathrm{loc}}((0,+\infty),\X)$, and every initial condition $U_{0}\in\X$, problem (\ref{eqn:abstract})--(\ref{eqn:datum}) admits a unique mild solution $U\in C^{0}([0,+\infty),\X)$, given by the formula
\begin{equation}
U(t)=S(t)U_{0}+\int_{0}^{t}S(t-\tau)F(\tau)\,d\tau
\qquad
\forall t\geq 0.
\label{th:mild-positive}
\end{equation}

If in addition the forcing term $F$ is eventually bounded, namely
\begin{equation}
\limsup_{t\to +\infty}\|F(t)\|_{\X}<+\infty,
\label{hp:F-limsup}
\end{equation}
then the solution given by (\ref{th:mild-positive}) is bounded in $\X$.

\item  \emph{(Bounded solutions for all times).} For every $F\in L^{\infty}(\re,\X)$ there exists a unique mild solution to equation (\ref{eqn:abstract}) that is globally bounded (both for positive and negative times). This solution is given by the formula
\begin{equation}
U(t)=\int_{0}^{+\infty}S(\tau)F(t-\tau)\,d\tau
\qquad
\forall t\in\re.
\label{th:mild-global}
\end{equation}

If in addition $F$ is periodic, then the solution given by (\ref{th:mild-global}) is periodic as well.

\end{enumerate}
\end{thmbibl}

In the sequel we restrict ourselves to forcing terms with values in the subspace $\Y$, and we investigate the extent to which a bound on the norm in $\X$ of the forcing term $F(t)$ yields a bound on the seminorm $p$ of solutions $U(t)$. In particular, in the case of solutions defined for positive times we are interested in estimates such as 
\begin{equation}
\limsup_{t\to +\infty}p(U(t))\leq K\limsup_{t\to +\infty}\|F(t)\|_{\X},
\label{defn:UB}
\end{equation}
while in the case of solutions that are globally bounded or periodic we are interested in estimates such as
\begin{equation}
\sup_{t\in\re}p(U(t))\leq K\|F(t)\|_{L^{\infty}(\re,\X)},
\label{defn:GB}
\end{equation}
or even the ``time 0'' variant
\begin{equation}
p(U(0))\leq K\|F(t)\|_{L^{\infty}(\re,\X)}.
\label{defn:GB-0}
\end{equation}

We refer to estimates of the form (\ref{defn:UB}) as ``ultimate bounds'', and we refer to estimates of the form (\ref{defn:GB}) as ``global bounds'', or ``periodic bounds'' if the forcing term is also periodic. 

\begin{defn}[Optimal bounds]\label{defn:BC}
\begin{em}

Let us consider equation (\ref{eqn:abstract}) under the functional setting described above. 

\begin{itemize}

\item  The optimal ultimate bound $\KUB(\A,\Y,p)$ is the smallest constant $K$ for which (\ref{defn:UB}) holds true for every forcing term $F\in L^{1}_{\mathrm{loc}}((0,+\infty),\Y)$ satisfying (\ref{hp:F-limsup}), and every corresponding solution $U(t)$ to (\ref{eqn:abstract}) given by (\ref{th:mild-positive}).

\item  The optimal global bound $\KGB(\A,\Y,p)$ is the smallest constant $K$ for which (\ref{defn:GB}) holds true for every forcing term $F\in L^{\infty}(\re,\Y)$, and every corresponding solution $U(t)$ to (\ref{eqn:abstract}) given by (\ref{th:mild-global}).

\item  The optimal periodic bound $\KPB(\A,\Y,p)$ is the smallest constant $K$ for which (\ref{defn:GB}) holds true for every periodic forcing term $F\in L^{\infty}(\re,\Y)$, and every corresponding periodic solution $U(t)$ to (\ref{eqn:abstract}) given by (\ref{th:mild-global}).

\item  The ``time 0'' bounds $\KGB^{0}(\A,\Y,p)$ and $\KPB^{0}(\A,\Y,p)$ are defined in analogy to $\KGB(\A,\Y,p)$ and $\KPB(\A,\Y,p)$, just starting with the ``time 0'' inequality (\ref{defn:GB-0}) instead of (\ref{defn:GB}).

\end{itemize}

\end{em}
\end{defn}


\subsection{Equivalence of optimal bounds}

The main and somewhat surprising result of this section is the equivalence between the different notions of optimal bounds.

\begin{thm}[Equivalence of optimal bounds]\label{thm:abstract}

Let us consider equation (\ref{eqn:abstract}) under the functional setting described above. 

Then the optimal bounds introduced in Definition~\ref{defn:BC} are equal, namely
\begin{equation}
\KUB(\A,\Y,p)=\KGB(\A,\Y,p)=\KGB^{0}(\A,\Y,p)=\KPB(\A,\Y,p)=\KPB^{0}(\A,\Y,p).
\nonumber
\end{equation}

\end{thm}

\begin{proof}

Since $\A$, $\Y$ and $p$ are fixed, for the sake of simplicity we drop the dependence on them in the constants.

The equivalence of $\KGB$ and $\KGB^{0}$ follows from the invariance of (\ref{eqn:abstract}) by time-translations, meaning that if $U(t)$ is the solution corresponding to some forcing term $F(t)$, then for every $t_{0}\in\re$ it turns out that $U(t+t_{0})$ is the solution corresponding to $F(t+t_{0})$. For the same reason, $\KPB$ is equal to $\KPB^{0}$.

Therefore, it is enough to prove that $\KPB\leq\KUB\leq\KGB\leq\KPB$.

\subparagraph{\textmd{\textit{Inequality $\KPB\leq\KUB$.}}}

Let $F:\re\to\Y$ be any forcing term that is periodic and essentially bounded, and let $U(t)$ be the corresponding periodic solution to (\ref{eqn:abstract}) given by (\ref{th:mild-global}). Then it turns out that
\begin{equation}
\sup_{t\in\re}p(U(t))=
\limsup_{t\to +\infty}p(U(t))\leq
\KUB\limsup_{t\to +\infty}\|F(t)\|_{\X}=
\KUB\|F\|_{L^{\infty}(\re,\X)},
\nonumber
\end{equation}
where the two equalities follow from the periodicity of $U(t)$ and $F(t)$, respectively, and the inequality follows from the definition of $\KUB$ once that we regard $F(t)$ and $U(t)$ as functions defined for nonnegative times.

This proves the required inequality.

\subparagraph{\textmd{\textit{Inequality $\KUB\leq\KGB$.}}}

Let $F\in L^{1}_{\mathrm{loc}}((0,+\infty),\X)$ be a forcing term satisfying (\ref{hp:F-limsup}), and let $U(t)$ be a corresponding solution to (\ref{eqn:abstract}). For every $\ep>0$, let $T_{\ep}$ be such that
\begin{equation}
\|F(t)\|_{\X}\leq\ep+\limsup_{t\to +\infty}\|F(t)\|_{\X}
\nonumber
\end{equation}
for almost every $t\geq T_{\ep}$. Let $F_{\ep}\in L^{\infty}(\re,\X)$ be defined by
\begin{equation}
F_{\ep}(t):=\left\{
\begin{array}{l@{\qquad}l}
F(t) & \mbox{if }t\geq T_{\ep}, \\[0.5ex]
0 & \mbox{if }t< T_{\ep},
\end{array}\right.
\nonumber
\end{equation}
and let $U_{\ep}(t)$ be the unique globally bounded solution corresponding to $F_{\ep}(t)$.

The function $U(t)-U_{\ep}(t)$ is a solution to the homogeneous equation in the half-line $t\geq T_{\ep}$, and therefore from (\ref{hp:p-P}) and (\ref{hp:St}) we deduce that
\begin{equation}
\lim_{t\to +\infty}p(U(t)-U_{\ep}(t))\leq
P\cdot\lim_{t\to +\infty}\|U(t)-U_{\ep}(t)\|_{\X}=0.
\nonumber
\end{equation}

At this point we conclude that
\begin{eqnarray*}
\limsup_{t\to +\infty}p(U(t)) & \leq & 
\limsup_{t\to +\infty}\left\{p(U(t)-U_{\ep}(t))+p(U_{\ep}(t))\strut\right\}  \\[0.5ex]
& = &  \limsup_{t\to +\infty}p(U_{\ep}(t))  \\[0.5ex]
& \leq &  \sup_{t\in\re}p(U_{\ep}(t))  \\[0.5ex]
& \leq &  \KGB\|F_{\ep}\|_{L^{\infty}(\re,\X)}  \\[0.5ex]
& = &  \KGB\|F\|_{L^{\infty}((T_{\ep},+\infty),\X)}  \\[0.5ex]
& \leq & \KGB\left(\ep+\limsup_{t\to +\infty}\|F(t)\|_{\X}\right).
\end{eqnarray*}

Letting $\ep\to 0^{+}$ we obtain the required inequality.

\subparagraph{\textmd{\textit{Inequality $\KGB\leq\KPB$.}}}

Let $F\in L^{\infty}(\re,\X)$ be a bounded forcing term, and let $U(t)$ denote the corresponding solution to (\ref{eqn:abstract}) that is globally bounded. For every positive real number $T$, let $F_{T}\in L^{\infty}(\re,\X)$ denote the $T$-periodic function that coincides with $F(t)$ for $t\in[0,T)$, and let $U_{T}(t)$ denote the corresponding periodic solution to (\ref{eqn:abstract}). From (\ref{th:mild-global}) we know that
\begin{equation}
U(0)-U_{T}(0)=
\int_{0}^{+\infty}S(\tau)(F(-\tau)-F_{T}(-\tau))\,d\tau=
\int_{T}^{+\infty}S(\tau)(F(-\tau)-F_{T}(-\tau))\,d\tau,
\nonumber
\end{equation}
so that from (\ref{hp:St}) we deduce that
\begin{eqnarray*}
\|U(0)-U_{T}(0)\|_{\X} & \leq & 
\int_{T}^{+\infty}\|S(\tau)(F(-\tau)-F_{T}(-\tau))\|_{\X}\,d\tau  \\
& \leq & 
2\|F\|_{L^{\infty}(\re,\X)}\int_{T}^{+\infty}Ce^{-\delta\tau}\,d\tau,
\end{eqnarray*}
and in particular $U(0)-U_{T}(0)\to 0$ in $\X$ as $T\to +\infty$.

At this point from (\ref{hp:p-P}) we deduce that
\begin{eqnarray*}
p(U(0)) & \leq & 
p(U_{T}(0))+p(U(0)-U_{T}(0))  \\[0.5ex]
& \leq & \KPB^{0}\|F_{T}\|_{L^{\infty}(\re,\X)}+P\|U(0)-U_{T}(0)\|_{\X}  \\[0.5ex]
& \leq & \KPB^{0}\|F\|_{L^{\infty}(\re,\X)}+P\|U(0)-U_{T}(0)\|_{\X}.
\end{eqnarray*}

Letting $T\to +\infty$ we conclude that $\KGB^{0}\leq\KPB^{0}$, and hence also $\KGB\leq\KPB$.
\end{proof}

\begin{rmk}[Almost periodic forcing terms]
\begin{em}

In many applications one has to deal with almost periodic sources rather than general bounded ones. The importance of this special class has been underlined in many articles and specialized monographs devoted to wave phenomena, see for example~\cite{amerio-prouse, levitan-zhikov}. For this reason, one could introduce ``optimal almost periodic bounds'' in analogy with what we did in Definition~\ref{defn:BC}. Of course this notion would coincide with the other ones, since the class of almost periodic forcing terms is intermediate between the periodic and the bounded ones.

\end{em}
\end{rmk}

\begin{rmk}[Regular forcing terms]
\begin{em}

One could define optimal bounds by limiting oneself to forcing terms that are more regular, for example continuous or even of class $C^{\infty}$. Also with this restriction one ends up with the same constants. The reason is that solutions to (\ref{eqn:abstract}) depend in a continuous way on $F(t)$, in the sense that if $F_{n}(t)\to F_{\infty}(t)$ in $L^{1}_{\mathrm{loc}}$, then the sequence $U_{n}(t)$ of corresponding solutions converges to the limit solution $U_{\infty}(t)$ uniformly on compact time intervals.

\end{em}
\end{rmk}


\subsection{Application to velocity bounds for second order equations}

In this subsection we specialize the abstract theory developed so far to the case of velocity estimates for solutions to (\ref{eqn:scalar}) and (\ref{eqn:vector}).

\paragraph{\textmd{\textit{The scalar equation}}}

Let us consider the scalar ordinary differential equation (\ref{eqn:scalar}). 
It is well known that this equation can be written as a first order system
\begin{equation}
\left(
\begin{array}{@{}c@{}}
u'(t) \\
v'(t) 
\end{array}
\right)+
\left(
\begin{array}{@{}cc@{}}
0 & -1 \\
b & c
\end{array}
\right)
\left(
\begin{array}{@{}c@{}}
u(t) \\
v(t) 
\end{array}
\right)=
\left(
\begin{array}{@{}c@{}}
0 \\
f(t) 
\end{array}
\right),
\nonumber
\end{equation}
and hence also as an abstract equation of the form (\ref{eqn:abstract}) with
\begin{equation}
\X:=\re^{2},
\qquad
\A:=\left(
\begin{array}{@{}cc@{}}
0 & -1 \\
b & c
\end{array}
\right),
\qquad
U(t):=\left(
\begin{array}{@{}c@{}}
u(t) \\
u'(t) 
\end{array}
\right),
\qquad
F(t):=\left(
\begin{array}{@{}c@{}}
0 \\
f(t) 
\end{array}
\right).
\nonumber
\end{equation}

We observe that $F(t)$ takes its values in the subspace $\Y:=\{0\}\times\re$. If we are interested in ultimate bounds on the velocity of the form
\begin{equation}
\limsup_{t\to +\infty}|u'(t)|\leq K\limsup_{t\to +\infty}|f(t)|.
\label{est:UB-scalar}
\end{equation}
or in global bounds of the form
\begin{equation}
|u'(t)|\leq K\|f(t)\|_{L^{\infty}(\re,\re)}
\qquad
\forall t\in\re,
\label{est:GB-scalar}
\end{equation}
then we can consider the seminorm $p$ in $\X$ defined by $p(u,v):=|v|$. In this way the common value $\operatorname{OB}(\A,\Y,p)$ of the constants of Theorem~\ref{thm:abstract} turns out to be the optimal constant for which (\ref{est:UB-scalar}) and (\ref{est:GB-scalar}) hold true. In particular it is the same for both estimates, and can be characterized in several different ways, as shown in Theorem~\ref{thm:abstract}.

\paragraph{\textmd{\textit{The vector equation}}}

Let $H$ be a (real) Hilbert space, and let $A$ be a self-adjoint linear operator on $H$ with dense domain $D(A)$, and satisfying the coercivity assumption (\ref{hp:A-coercive}). Let us consider equation (\ref{eqn:vector}), which can be written in the form
\begin{equation}
\left(
\begin{array}{@{}c@{}}
u'(t) \\
v'(t) 
\end{array}
\right)+
\left(
\begin{array}{@{}cc@{}}
0 & -I \\
A & cI
\end{array}
\right)
\left(
\begin{array}{@{}c@{}}
u(t) \\
v(t) 
\end{array}
\right)=
\left(
\begin{array}{@{}c@{}}
0 \\
f(t) 
\end{array}
\right),
\nonumber
\end{equation}
and hence also as an abstract equation of the form (\ref{eqn:abstract}) with
\begin{equation}
\X:=D(A^{1/2})\times H,
\qquad
\A:=\left(
\begin{array}{@{}cc@{}}
0 & -I \\
A & cI
\end{array}
\right),
\qquad
U(t):=\left(
\begin{array}{@{}c@{}}
u(t) \\
u'(t) 
\end{array}
\right),
\qquad
F(t):=\left(
\begin{array}{@{}c@{}}
0 \\
f(t) 
\end{array}
\right).
\nonumber
\end{equation}

In this setting mild solutions $U\in C^{0}([0,+\infty),\X)$ correspond to weak solutions in the class (\ref{defn:weak-sol}), and similarly for global solutions defined for every $t\in\re$.

We observe that $F(t)$ takes its values in the subspace $\Y:=\{0\}\times H$. If we consider in the phase space $\X$ the seminorm $p$ defined by $p(u,v):=\|v\|_{H}$, then the common value $\operatorname{OB}(\A,\Y,p)$ of the constants of Theorem~\ref{thm:abstract} is the optimal constant for which ultimate or global bounds on the velocity hold true.

\begin{rmk}[The role of the seminorm]
\begin{em}

Different choices of the seminorm $p$ lead to optimal bounds for different quantities. For example, in the scalar case the seminorm $p(u,v):=|u|$ leads to ultimate or global bounds on the solution $u(t)$, while in the vector case the seminorm $p(u,v):=(\|A^{1/2}u\|_{H}^{2}+\|v\|_{H}^{2})^{1/2}$ leads to ultimate or global bounds on the energy of solutions.

\end{em}
\end{rmk}


\setcounter{equation}{0}
\section{Optimal velocity bounds in the scalar case}\label{sec:scalar}

In this section we consider the scalar equation (\ref{eqn:scalar}), and we investigate the exact value of the constant that appears in optimal velocity bounds of the form (\ref{est:UB-scalar}) and (\ref{est:GB-scalar}). Since (\ref{eqn:scalar}) can be solved almost explicitly, we can compute the exact value of this constant, which we denote by $K(b,c)$. We also investigate the monotonicity and decay properties of $K(b,c)$ that will guide our exploration of the infinite dimensional case in the following section.

\begin{thm}[Optimal velocity bounds in the scalar case]\label{thm:scalar}

Let us consider equation (\ref{eqn:scalar}), where $b$ and $c$ are two positive real numbers. Let us set
\begin{equation}
\Delta:=\left|1-\frac{4b}{c^{2}}\right|^{1/2},
\label{defn:Delta}
\end{equation}
and let $K(b,c)$ denote the constant that appears in the optimal velocity bounds.  

Then it turns out that
\begin{equation}
K(b,c)=\left\{
\begin{array}{l@{\qquad}l}
\dfrac{2}{\sqrt{b}}\cdot\left(\dfrac{1-\Delta}{1+\Delta}\right)^{1/(2\Delta)} & 
\mbox{if }4b<c^{2}, \\[3ex]
\dfrac{4}{ec} & 
\mbox{if }4b=c^{2}, \\[3ex]
\dfrac{2}{\sqrt{b}}\cdot\left\{1-\exp\left(-\dfrac{\pi}{\Delta}\right)\right\}^{-1}\cdot\exp\left(-\dfrac{\arctan\Delta}{\Delta}\right) & 
\mbox{if }4b>c^{2}.
\end{array}
\right.
\label{th:kbc}
\end{equation}

As a consequence, the function $K(b,c)$ has the following properties.
\begin{itemize}

\item  \emph{(Monotonicity in $b$).} For every $c>0$, the function $b\to K(b,c)$ is decreasing.

\item  \emph{(Monotonicity in $c$)}. For every $b>0$, the function $c\to K(b,c)$ is decreasing.

\item  \emph{(Upper and lower bound).} It turns out that
\begin{equation}
\frac{4}{\pi c}<K(b,c)<\frac{2}{c}
\qquad
\forall (b,c)\in(0,+\infty)^{2},
\label{th:2/c}
\end{equation}
and the constants 2 and $4/\pi$ are optimal.
\end{itemize}

\end{thm}

\begin{rmk}\label{rmk:identity}
\begin{em}

The conclusions of Theorem~\ref{thm:scalar} hold true also if (\ref{eqn:scalar}) is interpreted as an evolution equation in a general Banach space $\X$. The proof relies on a standard duality argument. 

Let $\X'$ denote the dual of $\X$. For every $L\in\X'$ we consider the scalar functions $u_{*}(t):=L(u(t))$ and  $f_{*}(t):=L(f(t))$, and we observe that $u_{*}(t)$ solves a scalar equation of the form (\ref{eqn:scalar}) with forcing term $f_{*}(t)$, and therefore
\begin{equation}
|L(u'(0))|=
|u_{*}'(0)|\leq
K(b,c)\|f_{*}\|_{L^{\infty}(\re,\re)}\leq
K(b,c)\|f\|_{L^{\infty}(\re,\X)}\cdot\|L\|_{\X'}.
\nonumber
\end{equation}

Taking the supremum over all elements $L\in\X'$ with $\|L\|_{\X'}\leq 1$ we conclude that
\begin{equation}
\|u'(0)\|_{\X}\leq K(b,c)\|f\|_{L^{\infty}(\re,\X)}.
\nonumber
\end{equation}

Recalling the characterization of the optimal velocity bound as $\KGB^{0}$, this proves that the optimal velocity bound for solutions to (\ref{eqn:scalar}) in $\X$ is less than or equal to $K(b,c)$. 

The opposite inequality follows by considering ``simple modes'', namely solutions to (\ref{eqn:scalar}) in $\X$ of the form $v(t)=u(t)x_{0}$, where $x_{0}\in\X$ is any nonzero vector, and $u(t)$ is any solution to (\ref{eqn:scalar}) in $\re$.
 
\end{em}
\end{rmk}



\subsection{Proof of Theorem~\ref{thm:scalar} -- Computation of the optimal bound}

We use the characterization of $K(b,c)$ as $\KGB^{0}$, and we distinguish three cases according to the sign of $c^{2}-4b$, namely the discriminant of the characteristic equation
\begin{equation}
x^{2}+cx+b=0.
\label{eqn:char}
\end{equation}

\paragraph{\textmd{\textit{Non-oscillatory case}}}

When $4b<c^{2}$ we set
\begin{equation}
\alpha:=\frac{c+\sqrt{c^{2}-4b}}{2},
\qquad\qquad
\beta:=\frac{c-\sqrt{c^{2}-4b}}{2},
\nonumber
\end{equation}
so that the roots of the characteristic equation (\ref{eqn:char}) are the two negative real numbers $-\alpha$ and $-\beta$.

For every forcing term $f\in L^{\infty}(\re,\re)$, the unique solution to equation (\ref{eqn:scalar}) that is bounded for all (positive and negative) times is given by the formula
\begin{equation}
u(t):=\frac{1}{\alpha-\beta}\int_{0}^{+\infty}\left(-e^{-\alpha s}+e^{-\beta s}\right)f(t-s)\,ds
\qquad
\forall t\in\re.
\nonumber
\end{equation}

In particular it turns out that
\begin{equation}
u'(0)=\frac{1}{\alpha-\beta}\int_{0}^{+\infty}\left(\alpha e^{-\alpha s}-\beta e^{-\beta s}\right)f(-s)\,ds,
\label{formula:u'}
\end{equation}
and hence
\begin{equation}
|u'(0)|\leq\|f\|_{L^{\infty}(\re,\re)}\cdot\frac{1}{\alpha-\beta}\int_{0}^{+\infty}\left|\alpha e^{-\alpha s}-\beta e^{-\beta s}\right|\,ds.
\label{est:sup<}
\end{equation}

On the other hand, in the special case where
\begin{equation}
f(-s):=\operatorname{sign}\left(\alpha e^{-\alpha s}-\beta e^{-\beta s}\right)
\qquad
\forall s\in\re
\nonumber
\end{equation}
we find that
\begin{equation}
|u'(0)|=\|f\|_{L^{\infty}(\re,\re)}\cdot\frac{1}{\alpha-\beta}\int_{0}^{+\infty}\left|\alpha e^{-\alpha s}-\beta e^{-\beta s}\right|\,ds.
\label{est:sup>}
\end{equation}

This proves that $K(b,c)$ coincides with the constant that appears in right-hand side of both (\ref{est:sup<}) and (\ref{est:sup>}). In order to compute the integral, we observe that the integrand vanishes only in the point 
\begin{equation}
s_{0}:=\frac{1}{\alpha-\beta}\log\frac{\alpha}{\beta},
\label{defn:s-0}
\end{equation}
so that
\begin{eqnarray}
\int_{0}^{+\infty}\left|\alpha e^{-\alpha s}-\beta e^{-\beta s}\right|\,ds & = &
\int_{0}^{s_{0}}\left(\alpha e^{-\alpha s}-\beta e^{-\beta s}\right)\,ds-
\int_{s_{0}}^{+\infty}\left(\alpha e^{-\alpha s}-\beta e^{-\beta s}\right)\,ds 
\nonumber  \\
& = & \frac{2}{\alpha-\beta}\left(e^{-\beta s_{0}}-e^{-\alpha s_{0}}\right).
\nonumber
\end{eqnarray}

Now from (\ref{defn:s-0}) we obtain that
\begin{equation}
e^{-\beta s_{0}}-e^{-\alpha s_{0}}=
\left(\frac{\beta}{\alpha}\right)^{\beta/(\alpha-\beta)}-
\left(\frac{\beta}{\alpha}\right)^{\alpha/(\alpha-\beta)}=
\left(\frac{\beta}{\alpha}\right)^{(\alpha+\beta)/[2(\alpha-\beta)]}
\left(\sqrt{\frac{\alpha}{\beta}}-\sqrt{\frac{\beta}{\alpha}}\right),
\nonumber
\end{equation}
and therefore
\begin{equation}
\frac{2}{\alpha-\beta}\left(e^{-\beta s_{0}}-e^{-\alpha s_{0}}\right)=
\frac{2}{\sqrt{\alpha\beta}}\left(\frac{\beta}{\alpha}\right)^{(\alpha+\beta)/[2(\alpha-\beta)]}.
\nonumber
\end{equation}

Since
\begin{equation}
\alpha\beta=b,
\qquad\qquad
\alpha=\frac{c}{2}(1+\Delta),
\qquad\qquad
\beta=\frac{c}{2}(1-\Delta),
\nonumber
\end{equation}
we end up with the value given by (\ref{th:kbc}) in the case $4b<c^{2}$. 

\paragraph{\textmd{\textit{Critical case}}}

When $4b=c^{2}$ the characteristic equation (\ref{eqn:char}) has $-c/2$ as a root of multiplicity two. If $f\in L^{\infty}(\re,\re)$, equation (\ref{eqn:scalar}) admits a unique solution $u(t)$ that is globally bounded, and this solution is
\begin{equation}
u(t)=\int_{0}^{+\infty}se^{-cs/2}f(t-s)\,ds
\qquad
\forall t\in\re,
\nonumber
\end{equation}
so that
\begin{equation}
u'(0)=\int_{0}^{+\infty}e^{-cs/2}\left(1-\frac{cs}{2}\right)f(-s)\,ds.
\nonumber
\end{equation}

At this point the same argument of the non-oscillatory case shows that
\begin{equation}
K(b,c)=\int_{0}^{+\infty}e^{-cs/2}\left|1-\frac{cs}{2}\right|\,ds=\frac{4}{ec}.
\nonumber
\end{equation}

\paragraph{\textmd{\textit{Oscillatory case}}}

When $4b>c^{2}$ the characteristic equation (\ref{eqn:char}) has two complex conjugate roots of the form $-\gamma\pm\delta i$, where
\begin{equation}
\gamma:=\frac{c}{2},
\qquad\qquad
\delta:=\sqrt{b-\frac{c^{2}}{4}}.
\nonumber
\end{equation}

If $f\in L^{\infty}(\re,\re)$, equation (\ref{eqn:scalar}) admits a unique solution $u(t)$ that is globally bounded, and this solution is
\begin{equation}
u(t)=\frac{1}{\delta}\int_{0}^{+\infty}e^{-\gamma s}\sin(\delta s)f(t-s)\,ds
\qquad
\forall t\in\re,
\nonumber
\end{equation}
so that
\begin{equation}
u'(0)=\int_{0}^{+\infty}e^{-\gamma s}\left(\cos\left(\delta s\right)-\frac{\gamma}{\delta}\sin\left(\delta s\right)\right)f(-s)\,ds.
\nonumber
\end{equation}

Arguing again as in the non-oscillatory case, we obtain that
\begin{eqnarray*}
K(b,c) & = &  
\int_{0}^{+\infty}e^{-\gamma s}\left|\cos\left(\delta s\right)-\frac{\gamma}{\delta}\sin\left(\delta s\right)\right|\,ds  \\[1ex]
& = & 
\frac{1}{\delta}\int_{0}^{+\infty}e^{-\gamma y/\delta}\left|\cos y -\frac{\gamma}{\delta}\sin  y\right|\,dy.
\end{eqnarray*}

In order to compute the last integral, we call $g(y)$ the function inside the absolute value. Since $g(y)$ is $\pi$-periodic, we obtain that
\begin{eqnarray*}
\int_{0}^{+\infty}e^{-\gamma y/\delta}|g(y)|\,dy & = &
\sum_{k=0}^{\infty}\int_{k\pi}^{(k+1)\pi}e^{-\gamma y/\delta}|g(y)|\,dy  \\
& = & \left(\sum_{k=0}^{\infty}e^{-k\pi\gamma/\delta}\right)\int_{0}^{\pi}e^{-\gamma y/\delta}|g(y)|\,dy  \\
& = & \left(1-e^{-\pi\gamma/\delta}\right)^{-1}\int_{0}^{\pi}e^{-\gamma y/\delta}|g(y)|\,dy.
\end{eqnarray*}

Now we observe that $g(y)$, in the interval $(0,\pi)$, vanishes only in $y_{0}:=\arctan(\delta/\gamma)$, and
\begin{equation}
\int_{0}^{\pi}e^{-\gamma y/\delta}|g(y)|\,dy=
\int_{0}^{y_{0}}e^{-\gamma y/\delta}g(y)\,dy-
\int_{y_{0}}^{\pi}e^{-\gamma y/\delta}g(y)\,dy=
2\sin(y_{0})e^{-\gamma y_{0}/\delta}.
\nonumber
\end{equation}

Since
\begin{equation}
\frac{\gamma}{\delta}=\frac{1}{\Delta},
\qquad\qquad
y_{0}=\arctan\Delta,
\qquad\qquad
\delta=\frac{c\Delta}{2},
\nonumber
\end{equation}
and
\begin{equation}
\sin(y_{0})=\frac{\tan(y_{0})}{(1+\tan^{2}(y_{0}))^{1/2}}=\frac{\Delta}{(1+\Delta^{2})^{1/2}}=\frac{c\Delta}{2\sqrt{b}},
\nonumber
\end{equation}
we obtain the value of $K(b,c)$ given by (\ref{th:kbc}) in the case $4b>c^{2}$.


\subsection{Proof of Theorem~\ref{thm:scalar} -- Properties of the optimal bound}

From the explicit expression (\ref{th:kbc}), it is a calculus exercise to show that, for every $c>0$,
\begin{equation}
\lim_{b\to +\infty}K(b,c)=\frac{4}{\pi c}
\qquad\qquad\mbox{and}\qquad\qquad
\lim_{b\to 0^{+}}K(b,c)=\frac{2}{c}.
\nonumber
\end{equation}

Therefore, if we show that $K(b,c)$ is monotone with respect to $b$, this is enough to prove both (\ref{th:2/c}) and the optimality of the constants 2 and $4/\pi$.

Analogously, it is possible to show that $K(b,c)$ is continuous with respect to $b$ and with respect to $c$ in $(0,+\infty)^{2}$, since again the only nontrivial thing to check is the limit in the points with $4b=c^{2}$. 

Once we know the continuity, we can limit ourselves to show the monotonicity, both with respect to $b$ and with respect to $c$, in the two regions $0<4b<c^{2}$ and $4b>c^{2}$. For practical reasons, what we actually show is the monotonicity with respect to $b$ of the function $b\to c\cdot K(b,c)$, and the monotonicity with respect to $c$ of the function $c\to \sqrt{b}\cdot K(b,c)$.

\paragraph{\textmd{\textit{Monotonicity with respect to $b$ in the non-oscillatory regime}}}

When $c^{2}>4b$ it turns out that
\begin{eqnarray*}
c\cdot K(b,c) & = &
\frac{4}{(1-\Delta^{2})^{1/2}}\left(\frac{1-\Delta}{1+\Delta}\right)^{1/(2\Delta)} \\[1ex]
& = &
4\exp\left\{-\frac{1}{2}\log(1-\Delta^{2})+\frac{1}{2\Delta}\log\left(\frac{1-\Delta}{1+\Delta}\right)\right\},
\end{eqnarray*}
where $\Delta$ is defined by (\ref{defn:Delta}). When $b$ increases from 0 to $c^{2}/4$, the value of $\Delta$ decreases from $1$ to 0. Therefore, the function $b\to c\cdot K(b,c)$ is decreasing with respect to $b$ if and only if the function
\begin{equation}
f_{1}(x):=\frac{1}{x}\log\left(\frac{1-x}{1+x}\right)-\log(1-x^{2})
\nonumber
\end{equation}
is increasing with respect to $x$ in the interval $(0,1)$. Now let us set
\begin{equation}
g_{1}(x):=2x+\log\left(\frac{1-x}{1+x}\right).
\nonumber
\end{equation}

This function satisfies $g_{1}(0)=0$ and
\begin{equation}
g_{1}'(x)=\frac{2x^{2}}{x^{2}-1}<0
\qquad
\forall x\in(0,1),
\nonumber
\end{equation}
and therefore $g_{1}(x)<0$ for every $x\in(0,1)$. At this point we can conclude that
\begin{equation}
f_{1}'(x)=-\frac{g_{1}(x)}{x^{2}}>0
\qquad
\forall x\in (0,1),
\nonumber
\end{equation}
as required.

\paragraph{\textmd{\textit{Monotonicity with respect to $c$ in the non-oscillatory regime}}}

When $c$ increases from $2\sqrt{b}$ to $+\infty$, the value of $\Delta$ increases from 0 to 1. Since in this regime it turns out that
\begin{equation}
\sqrt{b}\cdot K(b,c)=2\left(\frac{1-\Delta}{1+\Delta}\right)^{1/(2\Delta)},
\nonumber
\end{equation}
we need to show that the function 
\begin{equation}
f_{2}(x):=\frac{1}{x}\log\left(\frac{1-x}{1+x}\right)
\nonumber
\end{equation}
is decreasing with respect to $x$ in the interval $(0,1)$. Now let us set
\begin{equation}
g_{2}(x):=\frac{2x}{x^{2}-1}-\log\left(\frac{1-x}{1+x}\right).
\nonumber
\end{equation}

This function satisfies $g_{2}(0)=0$ and
\begin{equation}
g_{2}'(x)=-\frac{4x^{2}}{(x^{2}-1)^{2}}<0
\qquad
\forall x\in(0,1),
\nonumber
\end{equation}
and therefore $g_{2}(x)<0$ for every $x\in(0,1)$. At this point we can conclude that
\begin{equation}
f_{2}'(x)=\frac{g_{2}(x)}{x^{2}}<0
\qquad
\forall x\in (0,1),
\nonumber
\end{equation}
as required.

\paragraph{\textmd{\textit{Monotonicity with respect to $b$ in the oscillatory regime}}}

Here we follow the argument introduced in~\cite{giraudo:tesi}. When $c^{2}<4b$ it turns out that $c\cdot K(b,c) =\varphi_{1}(\Delta)$, where
\begin{equation}
\varphi_{1}(x) :=
\frac{4}{(1+x^{2})^{1/2}}\exp\left(-\frac{\arctan x}{x}\right)\cdot\left\{1-\exp\left(-\frac{\pi}{x}\right)\right\}^{-1}.
\nonumber
\end{equation}

When $b$ increases from $c^{2}/4$ to $+\infty$, the value of $\Delta$ increases from $0$ to $+\infty$. Therefore, the function $b\to c\cdot K(b,c)$ is decreasing with respect to $b$ if and only if the function $\varphi_{1}(x)$ is decreasing with respect to $x$ in the half-line $(0,+\infty)$. A long but elementary computation gives that
\begin{equation}
\varphi_{1}'(x)=-4\psi_{1}(x)\cdot\left\{(e^{\pi/x}-1)(x-\arctan x)-\pi\right\}
\qquad
\forall x>0,
\nonumber
\end{equation}
where
\begin{equation}
\psi_{1}(x):=\exp\left(\frac{\pi-\arctan x}{x}\right)(e^{\pi/x}-1)^{-2}x^{-2}(x^{2}+1)^{-1/2}
\label{defn:psi-1}
\end{equation}
is a positive function. Therefore, $\varphi_{1}(x)$ is decreasing in $(0,+\infty)$ if and only if
\begin{equation}
(e^{\pi/x}-1)(x-\arctan x)>\pi
\qquad
\forall x>0.
\nonumber
\end{equation}

Setting $y:=\pi/x$ this inequality can be rewritten in the equivalent form
\begin{equation}
\frac{1}{\pi}\arctan\frac{\pi}{y}+\frac{1}{e^{y}-1}-\frac{1}{y}<0
\qquad
\forall y>0.
\label{est:giraudo-1}
\end{equation}

The second term in the left-hand side can be estimated from above by replacing $e^{y}$ with its Taylor polynomial of order four. In this way (\ref{est:giraudo-1}) is proved if we show that
\begin{equation}
\gamma(y):=\frac{1}{\pi}\arctan\frac{\pi}{y}+\frac{1}{y+\frac{y^{2}}{2}+\frac{y^{3}}{6}+\frac{y^{4}}{24}}-\frac{1}{y}<0
\qquad
\forall y>0.
\nonumber
\end{equation}

We observe that $\gamma(y)\to 0$ both when $y\to 0^{+}$ and when $y\to +\infty$. Therefore, it is enough to show that there exists $y_{0}>0$ such that $\gamma'(y)<0$ for every $y\in(0,y_{0})$ and $\gamma'(y)>0$ for every $y>y_{0}$. Another long but elementary computation shows that
\begin{equation}
\gamma'(y)=\frac{\pi^{2}(y^{4}+8y^{3}+40y^{2}+48y+48)-96(y^{3}+3y^{2}+6y+6)}{(y^{2}+\pi^{2})(y^{3}+4y^{2}+12y+24)^{2}}.
\nonumber
\end{equation}

The sign of $\gamma'(y)$ depends only on the sign of the numerator, which we denote by $\theta(y)$. Now we observe that $\theta(0)=48(\pi^{2}-12)<0$, and $\theta(y)>0$ when $y$ is large enough. Moreover, the second derivative
\begin{equation}
\theta''(y)=12\pi^{2}y^{2}+48(\pi^{2}-12)y+16(5\pi^{2}-36)
\nonumber
\end{equation}
is a polynomial of degree two with negative discriminant and positive leading coefficient. It follows that $\theta(y)$ is a convex function, and hence its sign switches from negative to positive at exactly one point $y_{0}>0$, as required.

\paragraph{\textmd{\textit{Monotonicity with respect to $c$ in the oscillatory regime}}}

When $c^{2}<4b$ it turns out that $\sqrt{b}\cdot K(b,c) =2\varphi_{2}(\Delta)$, where
\begin{equation}
\varphi_{2}(x) :=
\exp\left(-\frac{\arctan x}{x}\right)\cdot\left\{1-\exp\left(-\frac{\pi}{x}\right)\right\}^{-1}.
\nonumber
\end{equation}

When $c$ increases from $0$ to $2\sqrt{b}$, the value of $\Delta$ decreases from $+\infty$ to 0. Therefore, we need to show that the function $\varphi_{2}(x)$ is increasing with respect to $x$ in the half-line $(0,+\infty)$. Computing the derivative we find that
\begin{equation}
\varphi_{2}'(x)=\psi_{1}(x)\cdot(x^{2}+1)^{-1/2}\cdot\psi_{2}(x)
\qquad
\forall x>0,
\nonumber
\end{equation}
where $\psi_{1}(x)$ is the positive function defined in (\ref{defn:psi-1}), and
\begin{equation}
\psi_{2}(x):=e^{\pi/x}\left((x^{2}+1)\arctan x-x\right)+(x-\arctan x)+x^{2}(\pi-\arctan x)+\pi.
\nonumber
\end{equation}

At this point it is enough to verify that $\psi_{3}(x):=(x^{2}+1)\arctan x-x$ is positive for $x>0$, and this is true because $\psi_{3}(0)=0$ and $\psi_{3}'(x)=2x\arctan x>0$ for every $x>0$.


\setcounter{equation}{0}
\section{Optimal velocity bounds in the vector case}\label{sec:vector}

In this section we consider the vector equation (\ref{eqn:vector}), and we investigate the optimal constant, which now we call $K(A,c)$, involved in optimal ultimate bounds on the velocity  such as (\ref{th:alain-pde-u'}). As shown in section~\ref{sec:equivalence}, this constant can be characterized in several equivalent ways, including optimal global bounds on the velocity of the unique global solution that is bounded on the whole real line.  

The first result of this section is an upper bound for $K(A,c)$ that improves (\ref{th:alain-pde-u'}), at least when the damping is a positive multiple of the velocity. We point out that this upper bound tends to~0 as $c\to +\infty$.

\begin{thm}[Upper estimates for optimal velocity bounds]\label{thm:KcA-upper}

Let $H$ be a separable Hilbert space, let $A$ be a self-adjoint linear operator on $H$ with dense domain $D(A)$, and satisfying the coercivity assumption (\ref{hp:A-coercive}) for some positive real number $b$, and let $c$ be a positive number. Let $K(A,c)$ denote the constant that appears in the optimal velocity bounds for solutions to equation (\ref{eqn:vector}).

Then the following statements hold true.
\begin{enumerate}
\renewcommand{\labelenumi}{(\arabic{enumi})}

\item  \emph{(General case).}  Without any further restriction on $H$ and $A$ it turns out that
\begin{equation}
K(A,c)\leq\left\{
\begin{array}{l@{\qquad}l}
\dfrac{4}{c} & \mbox{if }c^{2}\leq 4b, \\[3ex]
\dfrac{4}{c}\sqrt{\log_{2}\left(\dfrac{c^{2}}{b}\right)} & \mbox{if }c^{2}> 4b.
\end{array}\right.
\label{th:KcA-upper}
\end{equation}

\item  \emph{(Finite dimensional case).}  If in addition the dimension of $H$ is a positive integer $d$, and hence $A$ is a positive symmetric $d\times d$ matrix, then it turns out that
\begin{equation}
K(A,c)<\frac{2}{c}\sqrt{d}
\qquad
\forall c>0.
\nonumber
\end{equation}

\end{enumerate}

\end{thm}

The second result of this section provides some lower bounds for $K(A,c)$, and shows that Theorem~\ref{thm:KcA-upper} is somewhat optimal. In particular, we prove that the factor $\sqrt{d}$ is essential in the finite dimensional case, that a decay of $K(A,c)$ of order $O(1/c)$ is impossible in the infinite dimensional case if $A$ is unbounded, and that the correction of order $(\log c)^{1/2}$ is essential when the sequence of eigenvalues of $A$ grows at most exponentially.

For the sake of simplicity, in the statement and in the proof we assume that $H$ admits an orthonormal basis of eigenvectors of $A$, but the result can be extended to general self-adjoint coercive operators by exploiting that they are unitary isomorphic to multiplication operators in suitable $L^{2}$ spaces (see for example~\cite[Theorem~VIII.4]{reed-simon}). We spare the reader this generality that only complicates proofs without introducing new ideas.


\begin{thm}[Lower estimates for boundedness constants]\label{thm:KcA-lower}

Let $H$ be a separable Hilbert space, and let $\{e_{i}\}_{i\in I}$ be an orthonormal basis of $H$, with either $I=\n$ or $I=\{1,\ldots,d\}$ for some integer $d\geq 1$. Let $\{\lambda_{i}\}_{i\in I}$ be a nondecreasing (finite or countable) sequence of positive real numbers, and let $A$ be the self-adjoint operator on $H$ such that
\begin{equation}
Ae_{i}=\lambda_{i}e_{i}
\qquad
\forall i\in I.
\nonumber
\end{equation}

Let $c$ be a positive real number, and let $K(A,c)$ denote the constant that appears in the optimal velocity bounds for solutions to equation (\ref{eqn:vector}).

Then the following statements hold true.
\begin{enumerate}
\renewcommand{\labelenumi}{(\arabic{enumi})}

\item  \emph{(General lower bound).} Without any further assumption it turns out that
\begin{equation}
K(A,c)\geq K(\lambda_{1},c)>\frac{4}{\pi c}
\qquad
\forall c>0,
\label{th:lower-basic}
\end{equation}
where $\lambda_{1}$ is the smallest eigenvalue of $A$, and $K(\lambda_{1},c)$ is the optimal velocity bound of the scalar case given by (\ref{th:kbc}). The constant $4/\pi$ in this lower bound is optimal.

\item  \emph{(Finite dimensional case).} Let us assume that the dimension of $H$ is a positive integer $d$. Then for every real number $\delta>0$ there exists a positive symmetric $d\times d$ matrix $A$ such that
\begin{equation}
K(A,c)\geq\frac{1-\delta}{c}\sqrt{d+3}
\nonumber
\end{equation}
when $c$ is large enough.

\item  \emph{(Unbounded operators).} Let us assume that $I=\n$ and $\lambda_{i}\to +\infty$ as $i\to +\infty$ (and therefore $A$ is an unbounded operator). Then it turns out that
\begin{equation}
\lim_{c\to +\infty}c\cdot K(A,c)=+\infty.
\label{th:lower-lim}
\end{equation}

\item  \emph{(Operators with eigenvalues growing at most exponentially).} There exists a real number $R_{0}>1$ with the following property.  If $I=\n$, and for some real number $R\geq R_{0}$ the sequence $\{\lambda_{n}\}_{n\in\n}$ admits a subsequence (not relabeled) such that
\begin{equation}
R_{0}\leq\frac{\lambda_{n+1}}{\lambda_{n}}\leq R
\qquad
\forall n\geq 1,
\label{hp:6R}
\end{equation}
then it turns out that
\begin{equation}
K(A,c)\geq\frac{1}{2(\log_{2}R)^{1/2}}\cdot\frac{(\log_{2}c)^{1/2}}{c}
\nonumber
\end{equation}
when $c$ is large enough.

\end{enumerate}

\end{thm}

We are now in a position to show that in the vector case the optimal velocity bound $K(A,c)$ does not depend in a decreasing way on the operator $A$, in contrast with what happens in the scalar case.

\begin{cor}[Lack of monotonicity of the optimal velocity bound]\label{cor:no-monotone}

If the dimension of $H$ is at  least two, then there exist two self-adjoint coercive operators $A_{1}$ and $A_{2}$ on $H$ with $A_{1}<A_{2}$ (in the sense that $A_{2}-A_{1}$ is a positive operator), but $K(A_{1},c)<K(A_{2},c)$ when $c$ is large enough.

\end{cor}\begin{proof}

To begin with, we consider the case where the dimension of $H$ is exactly two. Let $\delta$ be a positive real number such that $(1-\delta)\sqrt{5}\geq 2$. From statement~(2) of Theorem~\ref{thm:KcA-lower} we know that there exists a positive symmetric $2\times 2$ matrix $A_{2}$ such that $K(A_{2},c)\geq(1-\delta)\sqrt{5}/c$ when $c$ is large enough. Let $A_{1}$ be the identity $2\times 2$ matrix multiplied by a positive real number $b$. From Remark~\ref{rmk:identity} we know that $K(A_{1},c)=K(b,c)$, and therefore from (\ref{th:2/c}) and our definition of $\delta$ we deduce that
\begin{equation}
K(A_{1},c)<\frac{2}{c}\leq\frac{(1-\delta)\sqrt{5}}{c}\leq K(A_{2},c)
\nonumber
\end{equation}
when $c$ is large enough, and we conclude by observing that $A_{1}<A_{2}$ when $b$ is less than the smallest eigenvalue of $A_{2}$.

If the dimension of $H$ is greater than two, then we consider the operator $A_{1}^{*}$ that is again equal to $b$ times the identity, and the operator $A_{2}^{*}$ that coincides with $A_{2}$ in a two dimensional subspace $H^{*}$ of $H$, and with $(b+1)$ times the identity in the orthogonal of $H^{*}$. At this point the conclusion follows from the two dimensional case because $A_{1}^{*}<A_{2}^{*}$, but 
$$K(A_{1}^{*},c)=K(b,c)=K(A_{1},c)<K(A_{2},c)\leq K(A_{2}^{*},c),$$
as required. 
\end{proof}

Finally, we observe that assumption (\ref{hp:6R}) is satisfied whenever the eigenvalues of $A$ do not grow more than exponentially. The following remark shows that the assumption is satisfied in many applications to partial differential equations.

\begin{rmk}[Dirichlet Laplacian]\label{rmk:Laplacian}
\begin{em}

One of the main examples that fit into the abstract framework of (\ref{eqn:vector}) is the dissipative wave equation of the form
\begin{equation}
u_{tt}+cu_{t}-\Delta u=f(t,x)
\nonumber
\end{equation}
in some bounded open set $\Omega\subseteq\re^{d}$ with regular enough boundary and (for example) Dirichlet boundary conditions.

In this case the operator $A$ is the Dirichlet Laplacian, and from Weyl's law we know that the number $N(\lambda)$ of eigenvalues of $A$ in the interval $[0,\lambda]$ satisfies
\begin{equation}
N(\lambda)\sim\gamma\lambda^{d/2}
\qquad
\mbox{as }\lambda\to +\infty,
\nonumber
\end{equation}
where $\gamma$ is a positive constant that depends on the measure of $\Omega$. This distribution implies that, for every pair of real numbers $R>R_{0}>1$, the sequence of eigenvalues admits a subsequence satisfying (\ref{hp:6R}). Therefore, from statement~(1) of Theorem~\ref{thm:KcA-upper} and statement~(4) of Theorem~\ref{thm:KcA-lower} we deduce that in this case $K(A,c)$ decays, when $c\to +\infty$, as $O((\log c)^{1/2}/c)$, and not faster.

\end{em}
\end{rmk}


\subsection{Proof of Theorem~\ref{thm:KcA-upper}}

In this proof we exploit the characterization of $K(A,c)$ as $\KUB$. We start with the second statement, whose proof is rather short.

\paragraph{\textit{\textmd{Statement (2)}}}

We observe that the matrix $A$ can be diagonalized, and therefore any solution to (\ref{eqn:vector}) is a vector whose components $u_{i}(t)$ are solutions to scalar ordinary differential equations of the form 
\begin{equation}
u_{i}''(t)+cu_{i}'(t)+\lambda_{i}u_{i}(t)=f_{i}(t),
\label{eqn:components}
\end{equation}
where the $\lambda_{i}$'s are the eigenvalues of $A$. From the scalar case we know that
\begin{equation}
\limsup_{t\to +\infty}|u_{i}'(t)|\leq 
\frac{2}{c}\limsup_{t\to +\infty}|f_{i}(t)|\leq
\frac{2}{c}\limsup_{t\to +\infty}\|f(t)\|_{H},
\nonumber
\end{equation}
and therefore
\begin{equation}
\limsup_{t\to +\infty}\|u'(t)\|_{H}\leq
\limsup_{t\to +\infty}\left\{\sum_{i=1}^{d}|u_{i}'(t)|^{2}\right\}^{1/2}\leq
\frac{2}{c}\sqrt{d}\cdot\limsup_{t\to +\infty}\|f(t)\|_{H}.
\nonumber
\end{equation}

\paragraph{\textit{\textmd{Statement (1)}}}

To begin with, we consider the special case where the operator $A$ is bounded and $c$ is large, in which case the optimal velocity bound can be estimated in terms of the norm of the operator and its coercivity constant. This special case is going to play an important role in the proof for general unbounded operators.

\begin{lemma}[Coercive bounded operators in non-oscillatory regime]\label{prop:mM}

Let $H$, $A$ and $c$ be as in Theorem~\ref{thm:KcA-upper}, and let $f\in L^{\infty}((0,+\infty),H)$.  Let us assume in addition that there exist positive real numbers $m$ and $M$ such that $c^{2}\geq 4m$ and 
\begin{equation}
m\|x\|_{H}^{2}\leq\|A^{1/2}x\|_{H}^{2}\leq M\|x\|_{H}^{2}
\qquad
\forall x\in D(A^{1/2}).
\label{hp:AM}
\end{equation}

Then every solution to (\ref{eqn:vector}) satisfies
\begin{equation}
\limsup_{t\to +\infty}\|u'(t)\|_{H}\leq
\frac{1}{c}\left(1+\frac{\sqrt{3M}}{\sqrt{m}}\,\right)\cdot\limsup_{t\to +\infty}\|f(t)\|_{H}.
\label{th:u'-mM}
\end{equation}

\end{lemma}

\begin{proof}

When $c^{2}\geq 4m$, and the damping operator $B$ is $c$ times the identity, then assumption (\ref{hp:BA}) is satisfied with $C:=c/m$, and the constant in (\ref{th:alain-pde-u}) turns out to be $\sqrt{3/m}$. Therefore, from the boundedness assumption (\ref{hp:AM}) and estimate (\ref{th:alain-pde-u}) we deduce that
\begin{equation}
\limsup_{t\to +\infty}\|Au(t)\|_H\leq 
\sqrt{M}\limsup_{t\to +\infty}\|A^{1/2}u(t)\|_{H}\leq 
\frac{\sqrt{3M}}{\sqrt{m}}\limsup_{t\to +\infty}\|f(t)\|_{H}.
\label{est:Au}
\end{equation}

Now we write equation (\ref{eqn:vector}) in the form
\begin{equation}
u''(t)+cu'(t)=f(t)-Au(t),
\nonumber
\end{equation}
and we regard it as a first order equation in the unknown $u'(t)$. We deduce that
\begin{equation}
\limsup_{t\to +\infty}\|u'(t)\|_{H}\leq
\frac{1}{c}\limsup_{t\to +\infty}\left(\|f(t)\|_{H}+\|Au(t)\|_{H}\strut\right),
\nonumber
\end{equation}
which implies (\ref{th:u'-mM}) because of (\ref{est:Au}).
\end{proof}

We are now ready for a proof of (\ref{th:KcA-upper}), for which we distinguish two cases. 

If $c^{2}\leq 4b$, the result follows directly from Theorem~\ref{thmbibl:alain-PDE}. Indeed, when the dissipation operator $B$ is $c$ times the identity, assumption (\ref{hp:BA}) is satisfied with $C:=c/b$, so that from (\ref{th:alain-pde-u'}) we deduce that
\begin{equation}
\limsup_{t\to +\infty}\|u'(t)\|_{H}\leq
\max\left\{\frac{\sqrt{3}}{\sqrt{b}},\frac{3}{c\sqrt{2}}\right\}
\limsup_{t\to +\infty}\|f(t)\|_{H}\leq
\frac{4}{c}\limsup_{t\to +\infty}\|f(t)\|_{H},
\label{est:u'-oscill}
\end{equation}
which proves (\ref{th:KcA-upper}) in this case.

If $c^{2}>4b$, we take the nonnegative integer $k$ such that 
\begin{equation}
2^{k}<\frac{c^{2}}{4b}\leq 2^{k+1},
\nonumber
\end{equation}
and we partition the half-line $[b,+\infty)$ as the union of $k+1$ intervals of the form
\begin{equation}
I_{j}:=[2^{j}b,2^{j+1}b)
\qquad
\forall j\in\{0,1,\ldots,k\}
\nonumber
\end{equation}
and the half-line $I_{\infty}:=[2^{k+1}b,+\infty)$.

From the spectral theory for self-adjoint operators, we know that one can write the Hilbert space $H$ as a direct orthogonal sum of $A$-invariant subspaces
\begin{equation}
H=H_{0}\oplus H_{1}\oplus\cdots\oplus H_{k}\oplus H_{\infty}
\label{H-decomp}
\end{equation}
with the property that
\begin{equation}
2^{j}b\cdot\|x\|_{H}^{2}\leq
\langle Ax,x\rangle\leq 
2^{j+1}b\cdot\|x\|_{H}^{2}
\qquad
\forall x\in H_{j}
\nonumber
\end{equation}
if $0\leq j\leq k$, and
\begin{equation}
\langle Ax,x\rangle\geq\ 2^{k+1}b\cdot\|x\|_{H}^{2}\geq\frac{c^{2}}{4}\|x\|_{H}^{2}
\qquad
\forall x\in H_{\infty}\cap D(A),
\nonumber
\end{equation}

According to the decomposition (\ref{H-decomp}), we can write $A$ as the sum of $k+2$ operators $A_{j}$, the solution $u(t)$ as the sum of $k+2$ functions $u_{j}(t)$, and the forcing term $f(t)$ as the sum of $k+2$ forcing terms $f_{j}(t)$ in such a way that
\begin{equation}
u_{i}''(t)+cu_{j}'(t)+A_{j}u_{j}(t)=f_{j}(t)
\nonumber
\end{equation}
for every admissible value of the index $j$.

For every $0\leq j\leq k$, the operator $A_{j}$ satisfies the assumptions of Lemma~\ref{prop:mM} with $M/m= 2$, and therefore
\begin{equation}
\limsup_{t\to +\infty}\|u_{j}'(t)\|_{H}\leq
\frac{1+\sqrt{6}}{c}\limsup_{t\to +\infty}\|f_{j}(t)\|_{H}\leq
\frac{4}{c}\limsup_{t\to +\infty}\|f(t)\|_{H}.
\nonumber
\end{equation}

Moreover, the operator $A_{\infty}$ has coercivity constant greater than or equal to $c^{2}/4$, and therefore in analogy with (\ref{est:u'-oscill}) we obtain that
\begin{equation}
\limsup_{t\to +\infty}\|u_{\infty}'(t)\|_{H}\leq
\frac{4}{c}\limsup_{t\to +\infty}\|f_{\infty}(t)\|_{H}\leq
\frac{4}{c}\limsup_{t\to +\infty}\|f(t)\|_{H}.
\nonumber
\end{equation}

From these estimates we deduce that
\begin{eqnarray*}
\limsup_{t\to +\infty}\|u'(t)\|_{H} & \leq  &
\limsup_{t\to +\infty}\left\{\|u_{\infty}'(t)\|_{H}^{2}+
\sum_{j=0}^{k}\|u_{j}'(t)\|_{H}^{2}\right\}^{1/2}\\
& \leq & \frac{4}{c}(k+2)^{1/2}\limsup_{t\to +\infty}\|f(t)\|_{H},
\end{eqnarray*}
and we conclude by observing that 
\begin{equation}
k+2\leq\log_{2}\left(\frac{c^{2}}{4b}\right)+2=\log_{2}\left(\frac{c^{2}}{b}\right).
\nonumber
\end{equation}


\subsection{Proof of Theorem~\ref{thm:KcA-lower}}

The technical core of the proof is the following result.

\begin{lemma}\label{lemma:big-step}

Let $H$ be a Hilbert space, let $A$ be a self-adjoint operator on $H$, and let $n\geq 2$ be an integer. Let $L>\ep$ be two positive real numbers such that
\begin{equation}
e^{-\ep}-2e^{-L}>0.
\label{hp:L-ep}
\end{equation}

Let us assume that $A$ admits $n$ positive eigenvalues $\lambda_{1}$, \ldots, $\lambda_{n}$ such that 
\begin{equation}
\lambda_{i+1}\geq\frac{2L}{\ep}\lambda_{i}
\qquad
\forall i\in\{1,\ldots,n-1\}.
\label{defn:lambda-i}
\end{equation}

Then it turns out that
\begin{equation}
K(A,c)\geq\frac{e^{-\ep}-2e^{-L}}{c}\cdot\sqrt{n+3}
\qquad
\forall c\geq\left(\frac{4L}{\ep}\lambda_{n}\right)^{1/2}.
\label{th:construction}
\end{equation}

\end{lemma}

\begin{proof}

Due to the characterization of $K(A,c)$ as $\KGB^{0}$, it is enough to exhibit a piecewise constant forcing term $f\in L^{\infty}(\re,H)$ such that 
\begin{itemize}

\item  $\|f(t)\|_{H}\in\{0,1\}$ for every $t\in\re$,

\item  the unique solution $u(t)$ to (\ref{eqn:vector}) that is globally bounded satisfies
\begin{equation}
\|u'(0)\|_{H}\geq\frac{e^{-\ep}-2e^{-L}}{c}\cdot\sqrt{n+3}
\qquad
\forall c\geq\left(\frac{4L}{\ep}\lambda_{n}\right)^{1/2}.
\nonumber
\end{equation}

\end{itemize}

To this end, for every $i=1,\ldots,n$ we set 
\begin{equation}
\alpha_{i}:=\frac{c+\sqrt{c^{2}-4\lambda_{i}}}{2},
\qquad\qquad
\beta_{i}:=\frac{c-\sqrt{c^{2}-4\lambda_{i}}}{2}.
\nonumber
\end{equation}

We observe that $-\alpha_{i}$ and $-\beta_{i}$ are the two roots of the characteristic equation 
\begin{equation}
x^{2}+cx+\lambda_{i}=0,
\nonumber
\end{equation}
and they are negative real numbers because $L>\ep$ and hence $c^{2}> 4\lambda_{i}$. Moreover $\beta_{i}$ is increasing with $i$. 

Now we set $T_{0}:=+\infty$, and $T_{i}:=\ep\beta_{i}^{-1}$ for every $i=1,\ldots,n$. We observe that $-T_{1}$, \ldots, $-T_{n}$ is an increasing sequence of negative real numbers. For every $i=1$, \ldots, $n-1$, we consider the functions $f_{i}:\re\to\{0,-1\}$ defined by
\begin{equation}
f_{i}(t):=\left\{
\begin{array}{l@{\qquad}l}
-1 & \mbox{if }t\in(-T_{i-1},-T_{i}],  \\[0.5ex]
0 & \mbox{otherwise,}
\end{array}\right.
\nonumber
\end{equation}
and for $i=n$ we consider the function $f_{n}:\re\to\{0,-1,1\}$ defined by
\begin{equation}
f_{n}(t):=\left\{
\begin{array}{l@{\qquad}l}
-1 & \mbox{if }t\in(-T_{n-1},-T_{n}],  \\[0.5ex]
1 & \mbox{if }t\in(-T_{n},0],  \\[0.5ex]
0 & \mbox{otherwise.}
\end{array}\right.
\nonumber
\end{equation}

Finally, we define $f:\re\to H$ by
\begin{equation}
f(t):=\sum_{i=1}^{n}f_{i}(t)e_{i}
\qquad
\forall t\in\re,
\nonumber
\end{equation}
where $e_{i}$ is a unit eigenvector of $A$ corresponding to the eigenvalue $\lambda_{i}$. The function $f(t)$ vanishes for every positive time, and jumps among unit vectors for negative times, and hence $\|f(t)\|_{H}=1$ for every $t\leq 0$. 

Let $u(t)$ denote the unique solution to (\ref{eqn:vector}) that is globally bounded. This solution can be written in the form
\begin{equation}
u(t)=\sum_{i=1}^{n}u_{i}(t)e_{i},
\nonumber
\end{equation}
where $u_{i}(t)$ is the unique bounded solution to the scalar ordinary differential equation (\ref{eqn:components}). We claim that, when $c$ satisfies the condition in (\ref{th:construction}), it turns out that
\begin{equation}
u_{i}'(0)\geq\frac{e^{-\ep}-2e^{-L}}{c}
\qquad
\forall i\in\{1,\ldots,n-1\},
\label{est:ui'}
\end{equation}
and
\begin{equation}
u_{n}'(0)\geq\frac{2(e^{-\ep}-2e^{-L})}{c}.
\label{est:un'}
\end{equation}

Since numerators are positive because of (\ref{hp:L-ep}), from these claims it follows that
\begin{equation}
|u'(0)|=
\left\{\sum_{i=1}^{n}u_{i}'(0)^{2}\right\}^{1/2}\geq
\frac{e^{-\ep}-2e^{-L}}{c}\cdot\sqrt{n+3},
\nonumber
\end{equation}
which is exactly (\ref{th:construction}).

In order to prove the claims, we start from the usual formula (see (\ref{formula:u'}))
\begin{equation}
u_{i}'(0)=\frac{1}{\alpha_{i}-\beta_{i}}\int_{0}^{+\infty}\left(\alpha_{i}e^{-\alpha_{i}s}-\beta_{i}e^{-\beta_{i}s}\right)f_{i}(-s)\,ds.
\nonumber
\end{equation}

If $i\in\{1,\ldots,n-1\}$, from the definition of $f_{i}$ we obtain that
\begin{eqnarray*}
u_{i}'(0) & = & 
\frac{1}{\sqrt{c^{2}-4\lambda_{i}}}\int_{T_{i}}^{T_{i-1}}\left(-\alpha_{i}e^{-\alpha_{i}s}+\beta_{i}e^{-\beta_{i}s}\right)\,ds 
\\
& = & 
\frac{1}{\sqrt{c^{2}-4\lambda_{i}}}\left(e^{-\alpha_{i}T_{i-1}}-e^{-\alpha_{i}T_{i}}-e^{-\beta_{i}T_{i-1}}+e^{-\beta_{i}T_{i}}\right)
\end{eqnarray*}
(with the obvious agreement that $e^{-\infty}=0$), and therefore
\begin{equation}
u_{i}'(0)\geq\frac{1}{\sqrt{c^{2}-4\lambda_{i}}}\left(e^{-\beta_{i}T_{i}}-e^{-\alpha_{i}T_{i}}-e^{-\beta_{i}T_{i-1}}\right)
\qquad
\forall i\in\{1,\ldots,n-1\}.
\label{est:ui'0-base}
\end{equation}

If $i=n$, from the definition of $f_{n}$ we obtain that
\begin{equation}
u_{n}'(0)= 
\frac{1}{\sqrt{c^{2}-4\lambda_{n}}}
\left\{\int_{T_{n}}^{T_{n-1}}\left(-\alpha_{n}e^{-\alpha_{n}s}+\beta_{n}e^{-\beta_{n}s}\right)\,ds+\int_{0}^{T_{n}}\left(\alpha_{n}e^{-\alpha_{n}s}-\beta_{n}e^{-\beta_{n}s}\right)\,ds\right\},
\nonumber
\end{equation}
and therefore
\begin{eqnarray}
u_{n}'(0) & = & 
\frac{1}{\sqrt{c^{2}-4\lambda_{n}}}\left(2e^{-\beta_{n}T_{n}}-2e^{-\alpha_{n}T_{n}}+e^{-\alpha_{n}T_{n-1}}-e^{-\beta_{n}T_{n-1}}\right) 
\nonumber
\\[0.5ex]
& \geq & \frac{2}{\sqrt{c^{2}-4\lambda_{n}}}\left(e^{-\beta_{n}T_{n}}-e^{-\alpha_{n}T_{n}}-e^{-\beta_{n}T_{n-1}}\right).
\label{est:un'0-base}
\end{eqnarray}

Let us consider the three exponential terms in (\ref{est:ui'0-base}) and (\ref{est:un'0-base}). We claim that the first one (positive) dominates the two negative ones. To this end, let us estimate the three exponents.
\begin{itemize}

\item  From the definition of $T_{i}$ we obtain that $\beta_{i}T_{i}=\ep$.

\item  As for the second exponent, we observe that
\begin{equation}
\frac{\alpha_{i}T_{i}}{\ep}=
\frac{\alpha_{i}}{\beta_{i}}=
\frac{c+\sqrt{c^{2}-4\lambda_{i}}}{c-\sqrt{c^{2}-4\lambda_{i}}}=
\frac{\left(c+\sqrt{c^{2}-4\lambda_{i}}\right)^{2}}{4\lambda_{i}}\geq
\frac{c^{2}}{4\lambda_{i}},
\nonumber
\end{equation}
and therefore when $c^{2}\geq 4L\lambda_{n}/\ep$ we obtain that
\begin{equation}
\alpha_{i}T_{i}\geq\frac{\ep c^{2}}{4\lambda_{i}}\geq\frac{\ep c^{2}}{4\lambda_{n}}\geq L.
\nonumber
\end{equation}

\item  As for the third exponent, we observe that
\begin{equation}
\frac{\beta_{i}T_{i-1}}{\ep}=
\frac{\beta_{i}}{\beta_{i-1}}=
\frac{c-\sqrt{c^{2}-4\lambda_{i}}}{c-\sqrt{c^{2}-4\lambda_{i-1}}}=
\frac{c+\sqrt{c^{2}-4\lambda_{i-1}}}{c+\sqrt{c^{2}-4\lambda_{i}}}\cdot\frac{\lambda_{i}}{\lambda_{i-1}},
\nonumber
\end{equation}
and therefore from (\ref{defn:lambda-i}) we deduce that
\begin{equation}
\beta_{i}T_{i-1}\geq\ep\cdot\frac{c}{2c}\cdot\frac{\lambda_{i}}{\lambda_{i-1}}\geq L.
\nonumber
\end{equation}

\end{itemize}

From these estimates we conclude that
\begin{equation}
e^{-\beta_{i}T_{i}}-e^{-\alpha_{i}T_{i}}-e^{-\beta_{i}T_{i-1}}\geq 
e^{-\ep}-e^{-L}-e^{-L}
\qquad
\forall i\in\{1,\ldots,n\}.
\label{est:exp}
\end{equation}

Since $\sqrt{c^{2}-4\lambda_{i}}\leq c$, plugging (\ref{est:exp}) into (\ref{est:ui'0-base}) and (\ref{est:un'0-base}), we obtain (\ref{est:ui'}) and (\ref{est:un'}), respectively, as required.
\end{proof}

We are now ready to prove the four statements of Theorem~\ref{thm:KcA-lower}.

\paragraph{\textmd{\textit{Statement~(1)}}}

We can limit ourselves to forcing terms of the form $f(t):=\varphi(t)e_{1}$, where $e_{1}$ is the  element of the orthonormal basis corresponding to the smallest eigenvalue $\lambda_{1}$ of $A$, and $\varphi\in L^{\infty}(\re,\re)$. In this case the unique global solution to (\ref{eqn:vector}) has the form $u(t)=v(t)e_{1}$, where $v(t)$ is the unique global bounded solution to the scalar equation
\begin{equation}
v''(t)+cv'(t)+\lambda_{1}v(t)=\varphi(t).
\nonumber
\end{equation}

At this point the result follows from Theorem~\ref{thm:scalar}. 

\paragraph{\textmd{\textit{Statement~(2)}}}

Let us choose real numbers $L_{0}>\ep_{0}>0$ such that 
$$e^{-\ep_{0}}-2e^{-L_{0}}\geq 1-\delta.$$ 

This is possible whenever $\ep_{0}$ is small enough and $L_{0}$ is large enough. Let us consider the diagonal matrix whose eigenvalue are $(2L_{0}/\ep_{0})^{i}$ for $i=1,\ldots, d$. At this point the conclusion follows from Lemma~\ref{lemma:big-step}.

\paragraph{\textmd{\textit{Statement~(3)}}}

Let us set $\ep:=1$ and $L:=2$. Since in this case the sequence of the eigenvalues of $A$ is unbounded, for every $n$ we can always find eigenvalues $\lambda_{1}$, \ldots, $\lambda_{n}$ satisfying (\ref{defn:lambda-i}). Therefore, from Lemma~\ref{lemma:big-step} we deduce that
\begin{equation}
\liminf_{c\to +\infty}c\cdot K(A,c)\geq\frac{e-2}{e^{2}}\cdot\sqrt{n+3}.
\nonumber
\end{equation}

Since $n$ is arbitrary, this proves (\ref{th:lower-lim}).

\paragraph{\textmd{\textit{Statement~(4)}}}

Let us choose real numbers $L_{0}>\ep_{0}>0$ such that 
$$e^{-\ep_{0}}-2e^{-L_{0}}\geq\frac{1}{2},$$
and let us set
\begin{equation}
R_{0}:=\frac{2L_{0}}{\ep_{0}}.
\nonumber
\end{equation}

Let $\{\lambda_{n}\}$ denote the (sub)sequence of eigenvalues of $A$ satisfying (\ref{hp:6R}). For every positive integer $n$, let us set
\begin{equation}
\sigma_{n}:=\left(2R_{0}R^{n-1}\lambda_{1}\right)^{1/2},
\nonumber
\end{equation}
and let us choose a positive integer $n_{0}$ such that $\sigma_{n_{0}}\geq 1$ and
\begin{equation}
R^{n-1}\geq 2R_{0}\lambda_{1}
\qquad
\forall n\geq n_{0}.
\label{defn:n0}
\end{equation}

We claim that
\begin{equation}
K(A,c)\geq\frac{1}{2}\cdot\frac{(\log_{2} c)^{1/2}}{c}\cdot\frac{1}{(\log_{2}R)^{1/2}}
\qquad
\forall c\geq\sigma_{n_{0}}.
\nonumber
\end{equation}

Since $\sigma_{n}$ is increasing and tends to $+\infty$, this estimate is proved if we show that for every $n\geq n_{0}$ it turns out that
\begin{equation}
K(A,c)\geq\frac{1}{2}\cdot\frac{(\log_{2} c)^{1/2}}{c}\cdot\frac{1}{(\log_{2}R)^{1/2}}
\qquad
\forall c\in[\sigma_{n},\sigma_{n+1}].
\label{pre-th}
\end{equation}

To this end, let us consider any $n\geq n_{0}$. Due to the estimate from below in (\ref{hp:6R}), the eigenvalues $\lambda_{1}$, \ldots, $\lambda_{n}$ satisfy
\begin{equation}
\frac{\lambda_{i+1}}{\lambda_{i}}\geq\frac{2L_{0}}{\ep_{0}}
\qquad
\forall i\in\{1,\ldots,n-1\}.
\nonumber
\end{equation}

On the other hand, from the estimate from above in (\ref{hp:6R}) we deduce that
\begin{equation}
\lambda_{n}\leq R^{n-1}\lambda_{1},
\nonumber
\end{equation}
and hence
\begin{equation}
c\geq
\sigma_{n}=
(2R_{0}R^{n-1}\lambda_{1})^{1/2}\geq
(2R_{0}\lambda_{n})^{1/2}=
\left(\frac{4L_{0}}{\ep_{0}}\lambda_{n}\right)^{1/2}.
\nonumber
\end{equation}

Therefore, from Lemma~\ref{lemma:big-step} we obtain that
\begin{equation}
K(A,c)\geq
\frac{e^{-\ep_{0}}-2e^{-L_{0}}}{c}\cdot\sqrt{n+3}\geq
\frac{1}{2c}\sqrt{n}
\qquad
\forall c\geq\sigma_{n}.
\label{est:cak-1}
\end{equation}

Since in addition $c\leq\sigma_{n+1}$, from (\ref{defn:n0}) we deduce that
\begin{equation}
c^{2}\leq
\sigma_{n+1}^{2}=
2R_{0}\lambda_{1}R^{n}\leq
R^{2n-1}\leq R^{2n},
\nonumber
\end{equation}
and therefore $n\log_{2}R\geq\log_{2}c$ (we recall that $c\geq 1$ because $\sigma_{n_{0}}\geq 1$). Plugging this estimate into (\ref{est:cak-1}) we obtain (\ref{pre-th}), and this completes the proof.


\setcounter{equation}{0}
\section{Future perspectives and open problems}\label{sec:open}

In this section we mention some questions and open problems, inspired by this paper, that it could be interesting to investigate.

The first one concerns the estimates on $u(t)$, which we did not study here for the sake of shortness. In the linear scalar case the complete answer was already provided in~\cite{2005-AccXL-Haraux}, see Theorem~\ref{thmbibl:alain-LODE} in the introduction. An extension of that result to the vector case could lead to interesting applications to semilinear problems.

\begin{open}

Find optimal ultimate bounds for solutions $u(t)$ to the vector equation (\ref{eqn:vector}).

\end{open}

Limiting ourselves to scalar problems, now we know the optimal bounds both for $u(t)$ and for $u'(t)$ in the case of the linear equation (\ref{eqn:scalar}). We suspect that exactly the same bounds could apply also to solutions to the nonlinear equation (\ref{eqn:loud}) when the nonlinearity satisfies (\ref{hp:loud}). 

\begin{open}\label{open:loud}

Let us consider the optimal bounds on $u(t)$ provided by Theorem~\ref{thmbibl:alain-LODE}, and the optimal bounds on $u'(t)$ provided by Theorem~\ref{thm:scalar}. Do they remain true also for solutions to the general equation (\ref{eqn:loud}) under Loud's assumption (\ref{hp:loud})?

\end{open}

A positive answer to problem~\ref{open:loud} would be consistent with the intuitive idea that a bigger restoring force prevents solutions from growing too much, and therefore the linear case is the worst case scenario compatible with (\ref{hp:loud}). Some evidence of this effect is provided by Theorem~\ref{thm:scalar}, where we proved the monotonicity of $K(b,c)$ with respect to $b$. On the other hand, Corollary~\ref{cor:no-monotone} suggests that this is true only in the scalar case. 

A third open question concerns the monotonicity of optimal velocity bounds with respect to $c$ in the vector case.

\begin{open}

Is it true that the constant $K(A,c)$ of section~\ref{sec:vector} is decreasing with respect to $c$?

\end{open}

Concerning optimal estimates, it could be interesting to reduce the gap between the upper estimates of Theorem~\ref{thm:KcA-upper} and the lower estimates of Theorem~\ref{thm:KcA-lower}. The question arises in both finite and infinite dimensional frameworks. In the finite dimensional case we can state the question as follows.

\begin{open}

Let $d\geq 2$ be an integer. Determine the supremum of $c\cdot K(c,A)$ as $A$ ranges over all positive symmetric $d\times d$ matrices, or at least the
\begin{equation}
\sup\left\{\liminf_{c\to +\infty}c\cdot K(A,c): \mbox{$A$ is a $d\times d$ positive symmetric matrix}\right\}.
\label{open:sup}
\end{equation}

\end{open}

From the results of section~\ref{sec:vector} we know that (\ref{open:sup}) is at least $\sqrt{d+3}$ and at most $2\sqrt{d}$. Just for completeness, we remind that the infimum of $c\cdot K(c,A)$ as $A$ ranges over all positive symmetric $d\times d$ matrices is $4/\pi$ because (\ref{th:lower-basic}) is optimal. 

Finally, it could be interesting to extend Theorem~\ref{thm:KcA-upper} to general dissipation operators as in equation (\ref{eqn:PDE-AB}). The techniques of this paper could probably handle the case where $B$ is of the form $cA^{\alpha}$, or more generally the case where $B$ commutes with $A$. The most delicate point in the proof is when we decomposed the space in (\ref{H-decomp}). In the general case there is no guarantee that the subspaces are both $A$-invariant and $B$-invariant, and this makes the problem more challenging.

\begin{open}

Provide estimates from above for the constants involved in optimal bounds for the velocity of solutions to (\ref{eqn:PDE-AB}) under the assumptions of Theorem~\ref{thmbibl:alain-PDE}.

\end{open}


\subsubsection*{\centering Acknowledgments}

The first and third authors are members of the \selectlanguage{italian}``Gruppo Nazionale per l'Analisi Matematica, la Probabilità e le loro Applicazioni'' (GNAMPA) of the ``Istituto Nazionale di Alta Matematica'' (INdAM). 

\selectlanguage{english}



\end{document}